\documentclass[12pt]{amsart}
\usepackage{amscd}
\usepackage{mathrsfs}
\usepackage{a4wide}
\usepackage{color}
\usepackage{amssymb,amsmath,amsthm,amsfonts,latexsym}
\usepackage[utf8x]{inputenc}
\usepackage{bbm} 

\usepackage{amsmath,amsfonts,amssymb,amsthm} 
\usepackage{tikz}
\usepackage{hyperref}
\usepackage[all]{xy}

\newtheorem{theorem}{Theorem}[section]
\newtheorem{corollary}[theorem]{Corollary}
\newtheorem{lemma}[theorem]{Lemma}
\newtheorem{claim}[theorem]{Claim}

\newtheorem{ejem}[theorem]{Example}
\newtheorem{observation}[theorem]{Observation}

\newtheorem{proposition}[theorem]{Proposition}
\newtheorem{definition}[theorem]{Definition}
\newtheorem{question}[theorem]{Question}

\newcommand{\fin}{\mathbf{Fin}}
\newcommand{\B}{\mathcal{B}}
\newcommand{\C}{\mathcal{C}}
\newcommand{\I}{\mathcal{I}}

\newcommand{\F}{\mathcal{F}}

\newcommand{\fsig}{F_\sigma}
\newcommand{\cantor}{2^\mathbb{N}}
\newcommand{\baire}{\mathbb{N}^\mathbb{N}}

\def\finN{ [\mathbb{N}]^{<\omega}}
\def\infiN{[\mathbb{N}]^{\omega}}

\def\leukfrac#1/#2{\leavevmode
               \kern.1em
                \raise.9ex\hbox{\the\scriptfont0 ${}_#1$}
                \hskip -1pt\kern-.1em
                /\kern-.15em\lower.10ex\hbox{\the\scriptfont0 ${}_#2$}}

\makeatletter
\def\diam{\mathop{\operator@font diam}\nolimits}
\makeatother

\title{Bases and  selectors for  tall families}

\begin{document}

\author{Jan Greb\'ik}
\address{Institute of Mathematics\\
Academy of Sciences of the Czech Republic\\
\v Zitn\'a 609/25  \\
 110 00 Praha 1-Nov\'e M\v esto, Czech Republic}\email{grebik@math.cas.cz}

\author{Carlos Uzc\'ategui}
\address{Escuela de Matem\'aticas, Universidad Industrial de Santander, Cra. 27 Calle 9 UIS Edificio 45\\  Bucaramanga, Colombia}\email{cuzcatea@saber.uis.edu.co}

\thanks{{The first-listed author was supported by the GACR project 15-34700L and RVO:
67985840.}{}}

\begin{abstract}
We show that the Nash-Williams theorem has a uniform version and that the Galvin theorem does not. We show that there is an $F_\sigma$ tall ideal on $\mathbb{N}$ without a Borel selector and also construct a $\mathbf\Pi^1_2$ tall ideal on $\mathbb{N}$ without a tall closed  subset. 
\end{abstract}

\maketitle

\section{Introduction}

A family  $\mathcal{C}$ of subsets of $\mathbb{N}$ is tall if for every infinite $x\subseteq \mathbb{N}$ there is an infinite $y\in \mathcal{C}$ such that $y\subseteq x$.
We are interested in tall families $\mathcal{C}$ which are in addition definable as subsets of $2^{\mathbb{N}}$.
Take for example the set $hom(c)$ of all monochromatic subsets of $\mathbb{N}$ for some coloring $c:[\mathbb{N}]^2\to 2$.
This is, by Ramsey theorem, a tall family and moreover it is a closed subset of $\cantor$.
We deal with the question when we can effectively witness that a family $\mathcal{C}$ is tall i.e. when there is a Borel function $S:2^{\mathbb{N}}\to 2^{\mathbb{N}}$ such that for every infinite $x\in 2^{\mathbb{N}}$ is $S(x)\in \mathcal{C}$, $S(x)$ is infinite and $S(x)\subseteq x$.
We call such function $S$ a Borel selector for $\mathcal{C}$.
Note that if there is a Borel selector $S$ for $\mathcal{C}$ then $\mathcal{C}$ contains analytic subfamily which is also tall.
This leads to a natural basis problem of whether a given tall family $\mathcal{C}$ contains simpler tall subfamily $\mathcal{C}'\subseteq \mathcal{C}$.
By simpler we mean that $\mathcal{C}'$ is of lower complexity (for example closed) or is of specific form (for example $hom(c)$ for some coloring $c$). 

The main source of examples of tall families of subsets of $\mathbb{N}$ are tall Borel ideals on $\mathbb{N}$.
Recall that an ideal $\mathcal{I}$ is Kat\v etov below an ideal $\mathcal{J}$ ($\mathcal{I}\le_K \mathcal{J}$) if there is a function $f:\mathbb{N}\to\mathbb{N}$ such that $f^{-1}[x]\in \mathcal{J}$ for every $x\in \mathcal{I}$.
It is proved in \cite{GrebikHrusak2017} that having a Borel selector is closed upwards in the Kat\v etov order and if $\mathcal{I}$ is a tall Borel ideal with a Borel selector then there is a tall Borel ideal $\mathcal{J}$ such that $\mathcal{I}\not\le_{K}\mathcal{J}$.
All known examples of Borel ideals so far have a Borel selector (for random ideal $\mathcal{R}$ see \cite{HMTU2017} and for Solecki ideal $\mathcal{S}$ see \cite{GrebikHrusak2017}). We show that there is a $F_\sigma$ tall ideal without a Borel selector. 
The proof of this result is based on the following facts. Every $F_\sigma$ ideal can be coded by a closed collection of sets, i.e. by an element of the hyperspace $K(2^{\mathbb{N}})$. 
In \cite{GrebikHrusak2017} it is proved that the set of codes of tall $F_\sigma$ ideals is a $\mathbf\Pi^1_2-$complete subset of $K(\cantor)$.
To show that there is an $F_\sigma$ ideal without a selector we prove that the complexity of the set of codes of $F_\sigma$ ideals with a Borel selector is $\mathbf\Sigma^1_2$. However, it is an open question to find a concrete example of such  $F_\sigma$ ideal.

Another important source of examples are given by some well studied generalizations of $hom(c)$. Given a subset $\mathcal{O}$ of infinite subsets of $\mathbb{N}$, a set $x\subseteq \mathbb{N}$ is called $\mathcal{O}$-homogeneous, if either $[x]^\omega \subseteq \mathcal{O}$ or   $[x]^\omega \cap \mathcal{O}=\emptyset$. A well known theorem of Silver says that for every analytic  set $\mathcal{O}$ the collection $hom(\mathcal{O})$ of $\mathcal{O}$-homogeneous sets is tall.  When $\mathcal{O}$ is open (resp. clopen), the corresponding Ramsey result is called  Galvin's lemma \cite{GalvinPrikry} (resp. Nash-Williams' theorem \cite{nash-williams68}).  The existence of Borel selectors for families of the form $hom(\mathcal{O})$ is a consequence of the fact  that the corresponding Ramsey theorem holds uniformly. For instance, the fact that the Random ideal $\mathcal{R}$ has a Borel selector is due to the fact there is uniform approach of finding an infinite monochromatic subset of a given set $x\subseteq \mathbb{N}$ (or having a Borel proof of Ramsey theorem) \cite{HMTU2017}. Analogously, we show that Nash-Williams' theorem also has a uniform version and thus $hom(\mathcal{O})$ has a Borel selector for every clopen set $\mathcal{O}$. However,  we show there is an open set $\mathcal{O}$ such that $hom(\mathcal{O})$ does not have a Borel selector and therefore  Galvin's lemma does not admit a uniform version.  Ramsey  type theorems have been analyzed from a related but different complexity point of view. Solovay (\cite{Solovay78}) studied when $hom(\mathcal{O})$ contains an element which is hyperarithmetical on the code of  $\mathcal{O}$ (see also \cite{avigad}). 

We show that the basis problem also has a negative answer.  We construct a $\mathbf\Pi^1_2$ tall ideal $\I$ such that $hom(\mathcal{O})\not\subseteq \I$ for all open set $\mathcal{O}\subseteq \infiN$, in particular, $\I$ does not contain any tall closed subset. It is still an open question whether every tall Borel (analytic) ideal contains a closed tall subset.

\section{Preliminaries}
In this section we fix our notation, give some basic definitions and results that are later used.
We consider the natural isomorphism $\mathcal{P}(\mathbb{N})\approx 2^{\mathbb{N}}$ and view all  relations such as $\subseteq,\cap, [\_]^{<\omega}$, etc, as defined on $2^{\mathbb{N}}$ i.e. we use  $x\subseteq y$,  $x\cap y, [x]^{<\omega}$, etc, for $x,y\in 2^\mathbb{N}$. We use the standard descriptive set theoretic notions and notations (as in \cite{Kechris94}). The projective classes are denoted $\mathbf \Sigma^1_n$ and $\mathbf \Pi^1_n$. 

\begin{definition}
Let $\mathcal{C}\subseteq 2^{\mathbb{N}}$ be a tall family.
We say that $\mathcal{C}$ has a Borel selector, if there is a Borel function $S:2^{\mathbb{N}}\to 2^{\mathbb{N}}$ such that for every $x\in 2^{\mathbb{N}}$ 
\begin{itemize}
\item $S(x)\subseteq x$,
\item if $|x|$ is infinite then $|S(x)|$ is infinite,
\item $S(x)\in \mathcal{C}$.
\end{itemize}
\end{definition}
Note that we define the notion of a Borel selector only for tall families so if we say that $\mathcal{C}$ has a Borel selector it automatically means that $\mathcal{C}$ is tall.
We say that a family $\mathcal{C}$ is hereditary if $y\in\mathcal{C}$ whenever there is $x\in\mathcal{C}$ such that $y\subseteq x$. We say that $\mathcal{I}\subseteq 2^{\mathbb{N}}$ is an ideal on $\mathbb{N}$ if it is hereditary and it is closed under finite unions.
The following characterization of $F_\sigma$ ideals on $\mathbb{N}$ was given by Mazur \cite{mazur}. 
Recall that a map $\varphi:2^{\mathbb{N}}\to [0,\infty]$ is a \emph{lower-semicontinuous submeasure (lcsms)} if for all $x,y\in \mathbb{N}$
\begin{itemize}
\item $\varphi(\emptyset)=0$,
\item $x\subseteq y$ implies $\varphi(x)\le \varphi(y)$,
\item $\varphi(x\cup y)\le \varphi(x)+\varphi(y)$,
\item $\varphi(x)=\lim_{n\to\infty}\varphi(x\cap n)$.
\end{itemize}
Each lcsms $\varphi$ naturally corresponds to the  $F_\sigma$ ideal $Fin(\varphi):=\{x:\varphi(x)<\infty\}$.

\begin{theorem}[Mazur \cite{mazur}]
An ideal $\mathcal{I}$ is $F_\sigma$ if and only if there is lcsms $\varphi$ such that $\mathcal{I}=Fin(\varphi)$.
\end{theorem}

From this characterization one easily deduces (see for example \cite{GrebikHrusak2017}) the following Proposition which allows us to consider $K(2^\mathbb{N})$ the hyperspace of closed subsets of $2^\mathbb{N}$ as a space of codes of $F_\sigma$ ideals.
For $K\in K(2^\mathbb{N})$, let $\mathcal{I}_K$ be ideal generated by $K$, i.e. $x\in \mathcal{I}_K$ if and only if there is $y_0,...,y_{n-1}\in K$ such that $\bigcup_{i< n}y_i\subseteq x$. Clearly, $\I_K$ is $F_\sigma$.

\begin{proposition}\label{prop1}
For every $F_\sigma$ ideal $\mathcal{I}$ there is $K\in K(2^\mathbb{N})$ such that $\mathcal{I}=\mathcal{I}_K$.
\end{proposition}

In \cite{GrebikHrusak2017} it is proved that the set of codes of tall $F_\sigma$ ideals and the set of codes of tall $F_\sigma$ ideals containing the ideal $\fin$ are $\mathbf\Pi^1_2-$complete.

Next we state the combinatorial theorems (as presented in \cite{Todor2006}). 
Let $s,t\in [\mathbb{N}]^{<\omega}$. 
We write $s\sqsubseteq t$ when there is $n\in \omega$ such that
$s=t\cap \{0,1,\cdots, n\}$ and we say that $s$ is an initial
segment of $t$.

\begin{theorem}[Galvin]
\label{Galvin}
Let $\mathcal{F}\subseteq [\mathbb{N}]^{<\omega}$ and an infinite $x\in 2^{\mathbb{N}}$. Then there is an infinite $y\subseteq x$ such that one of the following holds
\begin{itemize}
\item for all $z\in [y]^{\omega}$ there is $s\in \mathcal{F}$ such that $s\sqsubseteq z$,
\item $[y]^{<\infty}\cap \mathcal{F}=\emptyset$.
\end{itemize}
\end{theorem}
We can think of $\mathcal{F}$ as a coloring of $[\mathbb{N}]^{<\omega}$ and put $hom(\mathcal{F})\subseteq 2^{\mathbb{N}}$ for the family of all $y$  that satisfy one of the conditions from the Galvin's theorem, such sets are called $\mathcal{F}$-homogeneous. 
It is clear that $hom(\F)$ is an hereditary tall collection. Moreover,  the family of of sets that satisfy the second condition is closed and hereditary and  the family of sets that satisfy the first condition is $\mathbf\Pi^1_1$.
We write $\mathbb{P}_2$ for the set of all those $\mathcal{F}\subseteq [\mathbb{N}]^{<\omega}$ such that first condition in Galvin's theorem is never satisfied.

A special type of coloring of $[\mathbb{N}]^{<\omega}$ are as follows. We say that $\mathcal{B}\subseteq [\mathbb{N}]^{<\omega}$ is a {\em front} on an infinite $x\in 2^{\mathbb{N}}$ if 
\begin{itemize}
\item every two elements of $\mathcal{B}$ are $\sqsubseteq$-incomparable,
\item every infinite $y\subseteq x$ has an initial segment in
$\mathcal{B}$.
\end{itemize}

\begin{theorem}[Nash-Williams]
\label{NWthm}
Let $\mathcal{B}$ be a front on $\mathbb{N}$ and $\mathcal{F}\subseteq \mathcal{B}$ then for every infinite $x\in 2^{\mathbb{N}}$ there is an infinite $y\subseteq x$ such that one of the following holds
\begin{itemize}
\item $[y]^{<\omega}\cap \mathcal{B}\subseteq \mathcal{F}$,
\item $[y]^{<\omega}\cap \mathcal{F}=\emptyset$.
\end{itemize}
\end{theorem}

Let $\F\subseteq\B$ as above,  it is easy to verify that $y\in hom(\mathcal{F})$ iff $y$ satisfies one of the conditions from the Nash-Williams' theorem.  Moreover, the family $hom(\mathcal{F})$ is easily seen to be closed, hereditary and tall. 

\begin{proposition}
\label{rep-of-closed}
For every closed, tall and hereditary $K\subseteq 2^{\mathbb{N}}$ there is $\mathcal{F}\subseteq [\mathbb{N}]^{<\omega}$ such that $hom(\mathcal{F})=K$.
\end{proposition}
\begin{proof}
Define $\mathcal{F}_K=\{s\in [\mathbb{N}]^{<\omega}:s\not\in K\}$.  We claim that $hom(\mathcal{F}_K)$ is equal to $\{y\in [\omega]^\omega: [y]^{<\infty}\cap \mathcal{F}_K=\emptyset\}$. Let $y\in hom(\mathcal{F}_K)$ and suppose $y$  satisfies the first condition in Galvin's theorem.  
Since $K$ is tall there is an infinite $z\subseteq y$ such that $z\in K$. As $y$ satisfies the first condition, there is $s\in \mathcal{F}_K$ such that $s\sqsubseteq z$ but since $K$ is hereditary we have $s\in K$ and this contradicts the definition of $\mathcal{F}_K$. 

It remains to check that $K=hom(\mathcal{F}_K)$.
Clearly $\subseteq$ holds.
For the opposite take $x\not\in K$.
Since $K$ is hereditary and closed there must be some $n\in\mathbb{N}$ such that $x\cap n \not\in K$ then we have $x\cap n\in \mathcal{F}_K$. Thus $x\not\in hom(\mathcal{F}_K)$. 
\end{proof}

\begin{proposition}
The set $\mathbb{P}_2$ is $\mathbf\Pi^1_2-$complete.
\end{proposition}
\begin{proof}
This is a generalization of previous argument. Let $\mathcal{T}$ be the set of all $K\in K(2^{\mathbb{N}})$ which generates a tall $F_\sigma$ ideal. As it was already mentioned $\mathcal{T}$ is $\mathbf\Pi^1_2-$complete (see \cite{GrebikHrusak2017}).

Consider the continuous map $\psi:K(2^{\mathbb{N}})\to \mathcal{P}(\mathbb{N}^{<\omega})$ given by
$$s\in \psi(K) \Leftrightarrow \forall x\in K \ s\not\subseteq x.$$
One may check that $\mathcal{T}=\psi^{-1}(\mathbb{P}_2)$ and the desired result follows since $\mathbb{P}_2$ is easily seen to be $\mathbf\Pi^1_2$.
\end{proof}

\section{Positive results}
In this section we prove the uniform version of the Nash-Williams's theorem. To state our theorem in full generality we must first introduce several definitions.
 
\subsection{Uniformly $p^+$, $q^+$\ and selective ideals}

Let $\mathcal{I}$ be an ideal on $\mathbb{N}$.
Recall that $\mathcal{I}^+=2^{\mathbb{N}}\setminus \mathcal{I}$.
We say that $\I$ is $q^+$ if for all $x\in \I^+$ and  every partition $\{s_n\}_{n<\omega}$
of $x$ into finite sets there is $y\subseteq x$ such that $y\in\I^+$ and $|y\cap x_n|\leq 1$ for all $n<\omega$.
It is $p^+$ if for every decreasing sequence $(x_n)_{n<\omega}$ of sets in $\I^+$ there is $x\in \I^+$ such that $x\setminus x_n$ is finite for all $n<\omega$.
It is {\em selective}, if for every decreasing sequence $(x_n)_{n<\omega}$ of sets in $\I^+$ there is $x\in \I^+$ such that $x/ n\subseteq  x_n$  for all $n\in x$.   
We are interested in the uniform versions of these notions.
We say that  a Borel ideal $\I$ is {\em uniformly selective} if there is a Borel function $F$ such that whenever $(x_n)_{n<\omega}$ is a decreasing sequence of sets in $\I^+$, then $F((x_n)_{n<\omega})=x$ is in $\I^+$ and $x/n\subseteq x_n$ for all $n\in x$.
In an analogous way, we define when an ideal is uniformly $p^+$ or $q^+$.

\begin{lemma}\label{Selective=q+p}
A Borel ideal $\I$ is uniformly selective iff it is uniformly $p^+$ and  $q^+$.
\end{lemma}

\proof Follow an standard proof of the  fact that an ideal es selective iff it is $p^+$ and  $q^+$ (see for instance \cite[Lemma 7.4]{Todor2010}).
\qed

\begin{theorem}
	\label{unif-p-q} Let $\I$ be a $\fsig$ ideal. Then
	\begin{itemize}
		\item[(i)] $\I$ is uniformly $p^+$.
		
		\item[(ii)] if $\I$ is $q^+$, then it is uniformly $q^+$.
	\end{itemize}
	
	In particular, every selective $\fsig$ ideal is uniformly
	selective.
	 
\end{theorem}

\proof
Let $\{s_k\}_{k<\omega}$ be some enumeration of $[\mathbb{N}]^{<\omega}$ and let $\mu$ be the lower semicontinuous submeasure such that
$\I =\{x\in2^{\mathbb{N}}: \mu(x)<\infty\}$.
First we claim that for
each $n\in \omega$ there is a Borel function $G_n:2^{\mathbb{N}}\rightarrow 2^{\mathbb{N}}$ such that for all $x\not\in \I$, $G_n(x)$ is a finite
subset of $x$ and $\mu(G_n(x))\geq n$.
Define $G_n(x)=\emptyset$ for $x\in \I$.
For $x\in \I^+$ let $G_n(x)=s_k$ where $k$ is the minimal index such that $s_k\subseteq x$ and $\mu(s_k)\ge n$.

(i) Let $(x_n)_{n<\omega}$ be a decreasing sequence of sets in $\I^+$.
Define $G((x_n)_{n<\omega})=\bigcup_{n<\omega} G_n(x_n)$.
Then $G$ is Borel and has the required property.
	
(ii) We define inductively sequence of Borel functions $\{F_n\}_{n<\omega}$ where $F_n:2^{\omega}\times ([\mathbb{N}]^{<\omega})^\omega\to [\mathbb{N}]^{<\omega}$ and for $(x,(t_i)_{i<\omega})\in 2^{\omega}\times ([\mathbb{N}]^{<\omega})^\omega$ we have
\begin{itemize}
\item $F_0(x,(t_i)_{i<\omega})=\emptyset$,
\item if $n>0$, $x\in\I^+$ and $(t_i)_{i<\omega}$ is a partition of $x$ then let $F_n(x,(t_i)_{i<\omega})=s_k$ where $k$ is the minimal index such that $s_k$ is a partial selector, $\mu(s_k)\ge n$ and $F_{n-1}(x,(t_i)_{i<\omega})\subseteq s_k$,
\item otherwise put $F_n(x,(t_i)_{i<\omega})=\emptyset$.
\end{itemize}
These are clearly Borel conditions and the functions are well defined since $\I$ is $q^+$.
Finally put $F(x,(t_i)_{i<\omega})=\bigcup_{n<\omega} F_n(x,(t_i)_{i<\omega})$.
\qed

\begin{corollary}
	\label{fin-unifor-selective}
	$\fin$ is uniformly selective.
\end{corollary}

Let $\mathcal A$ be an almost disjoint family of infinite subsets of
$\mathbb{N}$ and $\I(\mathcal{A})$ be the ideal generated by $\mathcal A$. By a
result of Mathias \cite{Mathias77}, $\I(\mathcal{A})$ is selective. It
is easy to verify  that when $\mathcal A$ is closed (as a subset of
$\cantor$), then  $\I(\mathcal{A})$ is $\fsig$. Hence from Theorem
\ref{unif-p-q} we get the following

\begin{corollary}
	Let $\mathcal A$ be a closed almost disjoint family.  Then $\I(\mathcal{A})$
	is uniformly selective.
\end{corollary}

A natural question is whether the previous result can be extended
to any ideal of the form $\I(\mathcal{A})$ for $\mathcal A$ an analytic
almost disjoint family or, more generally, to any selective
analytic ideal.

\subsection{Uniform Ramsey type theorems}
\label{uniformRamsey}

Recall that the lexicographic order $<_{lex}$ on $[\mathbb{N}]^{<\omega}$ is defined by $s<_{lex}t$ if $min(s{\triangle} t)\in s$.
Let $x\in 2^{\mathbb{N}}$ be infinite and $\mathcal{B}\subseteq [x]^{<\omega}$ be a front on $x$ then the restriction of $<_{lex}$ on $\mathcal{B}$ is a well-order and its order type is called the rank of $\B$ (denoted $rank(\mathcal{B})$).

For $\mathcal{F}\subseteq [\mathbb{N}]^{<\omega}$  we define
$
\overline{\mathcal F}=\{s\in[\mathbb{N}]:^{<\omega}\; s\sqsubseteq t\; \mbox{for some $t\in \mathcal{F}$}\}
$.

\begin{lemma}
\label{1hom}
Let $\B$ be a front and $\F\subseteq  \overline{\B}$. Let  $\widehat{\F}=\{s\in \finN:\;\exists t\in\F, \exists t'\in\B, \, t\sqsubseteq s\sqsubseteq t' \}$. Then  $x\in hom(\F)$ if and only if $[x]^{<\omega}\cap \F= \emptyset$ or $[x]^{<\omega}\cap \overline{\B}\subseteq  \overline{\F}\cup \widehat{\F}$. 
\end{lemma}

\proof 

Let $x\in hom(\F)$. Suppose the first item  in the conclusion of Theorem \ref{Galvin} holds. Let $s\subset x$ with $s\in \overline{\B}$ and put $y= s\cup \{n\in x:\; n> \max{s}\}$. Thus there is $t\in \F$ such that $t\sqsubset y$.  Hence $s\sqsubseteq t$ or $t\sqsubseteq s$. In either case, $s\in \overline{\F}\cup \widehat{\F}$. Conversely, suppose that 
$[x]^{<\omega}\cap \overline{\B}\subseteq  \overline{\F}\cup \widehat{\F}$ and let $y\in [x]^{<\omega}$.  Since $\B$ is a front, there is $t\in \B$ such that $t\sqsubset y$. Then $t\in \overline{\F} \cup \widehat{\F}$. Since $t\in \B$,  there is $s\sqsubseteq t$ with $s\in \F$.  Hence $x\in hom(\F)$. \qed

\begin{theorem}
\label{Local-uniform-galvin} 
Let $\B$ be a front on some set $z\in \infiN$  and $\I$ be a uniformly selective Borel ideal on $\omega$. There is a Borel map $S:2^{\overline \B}\times (\I^+\upharpoonright z)\rightarrow \I^+$ such that $S(\F,x)$ is a $\F$-homogeneous subset of $x$ for all $x\in \I^+$ and $x\subseteq  z$.
\end{theorem}

\proof  We may assume that $\B$ is a front on $\mathbb{N}$ and proceed by induction on $\alpha=rank(\B)$. If $rank(\B)=\omega$,
then  $\B= [B]^{1}$. Let $S(\F,x)=(\bigcup \F)\cap
y$, if $(\bigcup\F)\cap x\in\I^+$. Otherwise,
$S(\F,x)=x\setminus \bigcup \F$. Since $\I^+$ is Borel, then $S$ is a Borel function. 

Now suppose that the claim holds for all fronts on some set $z\in [\mathbb{N}]^{\omega}$ of rank less then $\alpha$.
For each $n\in \mathbb{N}$ and $\mathcal{F}\subseteq \B$,  let 
\[
\mathcal{F}_{\{n+1\}}=\{t\in\finN:\; n<\min(t)\;\&\; \{n\}\cup t\in \mathcal{F}\}.
\]
Observe that  $\B_{\{n+1\}}$ is a front on $x/(n+1)=\{m\in x:\; n<m\}$ with rank less than
$\alpha$ and the function
$$\Gamma:2^{\overline \B}\times \I^+\to \prod_{n\in\mathbb{N}} (2^{\overline{\B_{\{n\}}}}\times \I^+\upharpoonright (\mathbb{N}\setminus n))$$
where $\Gamma(\F,x)=((\F_{\{n\}},x\setminus n))_{n\in \mathbb{N}}$ is Borel.
By the inductive hypothesis there is Borel function
$$S:\prod_{n\in\mathbb{N}} (2^{\overline{\B_{\{n\}}}}\times \I^+\upharpoonright (\mathbb{N}\setminus n))\to \prod_{n\in\mathbb{N}} (I^+\upharpoonright (\mathbb{N}\setminus n))$$
that satisfies the conclusion of the theorem for each coordinate. Denote the composition of $\Gamma$, $S$ and projection to $n$-th coordinate as $S_n$.

We define a sequence of Borel functions $\{H_n\}_{n<\omega}$.
For $(\F,x)\in 2^{\overline \B}\times \I^+$ define inductively
\begin{itemize}
\item $H_0(\F,x)=x$,
\item $H_{n+1}(\F,x)=S_{n+1}(\F,x)$ if $n\in x$ otherwise $H_{n+1}(\F,x)=H_{n}(\F,x)$.
\end{itemize}
Since $\I$ is uniformly selective,  we can extract, in a Borel way, from the sequence $\{H_n(\F,x)\}_{n<\omega}$ a set $y\in \I^+$ such that
\[
y/(n+1)\subseteq  H_{n+1}(\F,x)\;\; \mbox{for all $n\in y$.}
\]
Lemma \ref{1hom} naturally provides the notion of $i$-homogeneous for $\F$ for $i=0,1$. Let
\[
y_i=\{n\in y:\; \mbox{$H_{n+1}(\F,x)$ is $i$-homogeneous for $\F_{\{n+1\}}$}\}.
\]
Then  $y_i$ is $i$-homogeneous for $\F$. In fact, for $i=0$,
let $t$ be a finite subset of $y_0$ and let $n=\min(t)$. Then
$t/(n+1)\subseteq  H_{n+1}(\F,x)$ as $n\in y$. Therefore $t/(n+1)\not\in \F_{\{n+1\}}$, as $H_{n+1}(\F,x)$
is $0$-homogeneous. Thus $t=\{n\}\cup t/(n+1)\not\in\F$.  Using Lemma \ref{1hom}, a similar argument works for  $i=1$.

By Lemma \ref{1hom}, being $i$-homogeneous for $\F$ is a Borel property, therefore  the function $y\mapsto (y_0,y_1)$ is Borel.  Since $y\in \I^+$, then at least
one of the sets $y_0$ or $y_1$ belongs to $\I^+$.
Let $S(\F, x)=y_0$ if $y_0\in \I^+$ and $y_1$, otherwise. As
$\I^+$ is Borel, we can pick in a Borel way the  alternative that
holds. Thus $S$ is Borel. 
\qed

Since $\fin$ is uniformly selective (corollary \ref{fin-unifor-selective}), we get the uniform version of Nash-Williams' theorem.

\begin{corollary}
\label{uniform-galvin} Let $\B$ be a front on $\mathbb{N}$. There is a Borel map $S:2^{\B}\times [\mathbb{N}]^{<\omega}\rightarrow [\mathbb{N}]^{<\omega}$ such that $S(\F,x)$ is a $\F$-homogeneous subset of $x$, for all $x\in [\mathbb{N}]^{<\omega}$ and all $\F\subseteq \B$.
\end{corollary}

\medskip

Using the front $[\mathbb{N}]^{n}$,  we get that the classical Ramsey's
theorem holds uniformly (the case $n=2$ appeared in
\cite{HMTU2017}).

\begin{corollary}
\label{efective-Ramsey} For each $n\in\mathbb{N}$, there is a Borel
function $S: 2^{[\mathbb{N}]^{n}}\times[\mathbb{N}]^{<\omega}\rightarrow [\mathbb{N}]^{<\omega}$ such that $S(\F, x)$ is an infinite subset of $x$ homogeneous for $\F\subseteq [\mathbb{N}]^{n}$.
\end{corollary}

\section{Negative results}

In this section we show that there is a tall $F_{\sigma}$ ideal without a Borel selector and deduce from this fact that there is no uniform version of Galvin's theorem. We also show that there is a $\mathbf\Pi^1_2$ tall ideal $\I$ such that $hom(\mathcal{F})\not\subseteq \I$ for every $\mathcal{F}\subseteq [\mathbb{N}]^{<\omega}$.

\subsection{A $F_\sigma$ ideal without a selector and no uniform version of Galvin's theorem}

Recall that the hyperspace $K(2^{\mathbb{N}})$ serves  as a space of codes for $F_\sigma$ ideals (see Proposition \ref{prop1}).
In \cite{GrebikHrusak2017} it is proved that the set of codes of tall $F_\sigma$ ideals is $\mathbf\Pi^1_2-$complete.
To show that there is an $F_\sigma$ ideal without a selector we prove that the complexity of the set of codes of $F_\sigma$ ideals with a Borel selector is $\mathbf\Sigma^1_2$.

We start by modifying a bit the notion of tallness and  Borel selector. For $K\in K(2^\mathbb{N})$, let 
$$
\downarrow K=\{x:\exists y\in K \ x\subseteq y\}.
$$
\begin{definition}
We say that $K\in K(2^\mathbb{N})$ is pseudo-tall if for every infinite $x\in 2^\mathbb{N}$ there is infinite $y\in \downarrow K$ such that $y\subseteq x$. 
\end{definition}

One can verify that as a function $\downarrow:K(2^\mathbb{N})\to K(2^\mathbb{N})$ is continuous and $K$ is pseudo-tall if and only if $\mathcal{I}_K$ is tall.

\begin{proposition}\cite{GrebikHrusak2017}
\label{pseudo1}
Given $K\in K(2^\mathbb{N})$, there is a Borel function $\phi:\mathcal{I}_K \to K^{<\omega}$ such that $x\subseteq\bigcup\phi(x)$.
\end{proposition}

\begin{proposition}
Let $K\in K(2^\mathbb{N})$ be pseudo-tall. Then $\mathcal{I}_K$ has a Borel selector $S$ if and only if there is a Borel selector $S'$ such that $rng(S')\subseteq \downarrow K$.
\end{proposition}

\begin{proof}
Using Proposition \ref{pseudo1}, it is enough to realize that if $x$ is infinite then at least one set in $\phi(x)$ must have infinite intersection with $x$ and since $\phi(x)$ is finite we can pick such a set in a Borel way.
\end{proof}

This leads to a modified definition of a selector.
\begin{definition}
Let $K\in K(2^\mathbb{N})$ be a pseudo-tall. We say that $K$ has a Borel \emph{pseudo-selector} if there is a Borel function $S:2^\mathbb{N}\to2^\mathbb{N}$ such that
\begin{itemize}
\item $S(x)\in\downarrow K$,
\item if $|x|=\omega$ then $|S(x)|=\omega$,
\item $S(x)\subseteq x$.
\end{itemize}
\end{definition}

By the previous proposition, $K\in K(2^\mathbb{N})$ has a pseudo-selector if and only if $\mathcal{I}_K$ has a selector and therefore it suffices to consider only  pseudo-selectors of closed subsets of $2^{\mathbb{N}}$, in other words the questions of existence of a Borel selector for $F_\sigma$ ideals and hereditary tall closed subsets of $2^\mathbb{N}$ are equivalent.
Let us summarize this in the following proposition.

\begin{proposition}
Let $K\in K(2^{\mathbb{N}})$ be tall. The following are equivalent:
\begin{itemize}
\item there is a Borel selector for $K$,
\item there is a Borel pseudo-selector for $K$,
\item the $F_\sigma$ ideal $\mathcal{I}_K$ has a Borel selector,
\item the smallest ideal $\mathcal{I}$ that contains $K$ and $\fin$ has a Borel selector.
\end{itemize}
\end{proposition}
\begin{proof}
It can be easily verified that the ideal $\mathcal{I}$ in the fourth condition is also $F_\sigma$. The only implication that is not clear from previou arguments is how to get a Borel selector from a Borel pseudo-selector.

Let $S:2^{\mathbb{N}}\to\mathbb{N}$ be a Borel pseudo-selector for $K$. Define
$$
\{(x,y):S(x)\subseteq y\subseteq x, \ y\in K\}\subseteq 2^{\mathbb{N}}\times \mathbb{N}.
$$
This is a Borel set with each vertical section compact and therefore it has a Borel uniformization by a classical uniformization theorem (see, for instance, \cite[Theorem 35.46]{Kechris94}). The uniformizing function  is a Borel selector for K.

\end{proof}

\subsubsection{Coding of Borel functions}

Now we are going to present how to code Borel functions. For that end, first we need to code Borel sets.  This coding is somewhat standard  (see for instance \cite[pag. 19]{Gao}), but we need to present it with some detail.  We define a set of labeled well-founded trees which will be the codes of Borel sets. 

\begin{definition}
Let $\mathcal{L}T$ be the set of all trees on $\mathbb{N}$ where each node is labeled by an element of $\{0,1\}$.
\end{definition}

So, formally, every element of $\mathcal{L}T$ is a tuple $(T,f)$ where $T\subseteq \mathbb{N}^{<\omega}$ is a tree and $f:T\to 2$. However, we will always write only $T\in \mathcal{L}T$ and $(s,i)\in T$ meaning that $f(s)=i$.

One can easily check that there $\mathcal{L}T$ is a closed subset of the Polish space of all trees on $\mathbb{N}\times 2$, thus $\mathcal{L}T$ is a Polish space. Moreover,  the set of all well-founded labeled trees $WF\mathcal{L}T$ is $\mathbf\Pi^1_1$.

We are interested in a closed subspace of $\mathcal{L}T$ which will  contain all codes for Borel subsets of $2^\mathbb{N}$.

\begin{definition}
Let $\mathcal{L}T_c\subseteq \mathcal{L}T$ be the set of all labeled trees satisfying the following condition.
\begin{itemize}
\item if $(s,1)\in T$ then $(s^\frown (0),0)\in T$ and it is the only immediate successor of $(s,1)$.
\end{itemize}
\end{definition}

One can easily verify that $\mathcal{L}T_c$ is a closed subspace of $\mathcal{L}T$ and the set of well-founded trees $WF\mathcal{L}T_c\subseteq \mathcal{L}T_c$ is $\mathbf\Pi^1_1$.

Now we will define, for each $T\in WF\mathcal{L}T_c$, the Borel set $A_T$ coded by $T$. And conversely, for each Borel set $A\subseteq \cantor$  there will be a $T\in WF\mathcal{L}T_c$ such that $A=A_T$.  The definition of $A_T$ is by recursion on the rank of $T$. 

Let $\{t_n:\;n\in\mathbb{N}\}$ be an enumeration of all basic open sets of $2^\mathbb{N}$, i.e. each $t_n$ is  a finite binary sequence.  Recursively define what each $(s,i)\in T$ codes:
\begin{itemize}
\item if $(s,0)$ is a leaf then it codes the basic open set $t_{s(|s|-1)}$ (in the case of $s=\emptyset$, we put $t_{\emptyset(|\emptyset|-1)}=t_0$),
\item if $(s,0)$ is not a leaf, then it codes the union of the sets coded by  $(s^\frown n,i)$ where $(s^\frown n,i)\in T$,
\item $(s,1)$ codes the complement of what $(s^\frown(0),0)$ codes.
\end{itemize}
Finally, $A_T$ is the set coded by $(\emptyset,i)$. 

\begin{lemma}
\label{coding}
For every Borel set $A\subseteq \cantor$ there is $T\in WF\mathcal{L}T_c$ such that $A=A_T$. And conversely, $A_T$ is Borel for each $T\in WF\mathcal{L}T_c$. 
\end{lemma}

\begin{proof}
Given $T\in WF\mathcal{L}T_c$, one easily shows for induction on the rank of $T$ that  $A_T$ is Borel. Conversely, given a Borel set $A\subseteq \cantor$, by induction on the Borel complexity of $A$ it is easy to construct a $T\in WF\mathcal{L}T_c$ such that $A=A_T$ 
\end{proof}

Let  $\mathcal{C}_i\subseteq 2^\mathbb{N}\times \mathcal{L}T_c$, $i\in 2$,  be given by 
\[
(x,T)\in\mathcal{C}_1\;\mbox{if and only if $T\in WF\mathcal{L}T_c$ and $x\in A_T$} 
\]
and
\[
(x,T)\in\mathcal{C}_0\;\mbox{if and only if $T\in WF\mathcal{L}T_c$ and $x\not\in A_T$.} 
\]

The following is a crucial result.

\begin{lemma}
\label{CDcoanalytic}
The relation $\mathcal{C}_i$ is $\mathbf \Pi^1_1$ for $i\in 2$.
\end{lemma}

For the proof we need some auxiliary results.  We define the following subset $G\subseteq 2^\mathbb{N}\times\mathcal{L}T_c\times\mathcal{L}T$.

\begin{definition}\label{G}
A triple $(x,T,S)$ is in $G\subseteq 2^\mathbb{N}\times\mathcal{L}T_c\times\mathcal{L}T$ if and only if
\begin{itemize}
\item $(s,i)\in T$ for some $i\in 2$ if and only if $(s,j)\in S$ for some $j\in 2$,
\item if $(s,0)\in T$ is leaf then $(s,1)\in S$ if and only if $t_{s(|s|-1)}\sqsubseteq x$,
\item if $(s,1)\in T$ then $(s,1)\in S$ if and only if $(s^\frown (0),0)\in S$,
\item if $(s,0)\in T$ not a leaf then $(s,1)\in S$ if and only if there is $n\in \mathbb{N}$ such that $(s^\frown n, 1)\in S$.
\end{itemize}
\end{definition}
Note that if $(x,T,S)\in G$ then $S$ has the same tree structure as $T$, it only has different labeling.
Also note that if $T$ is well-founded then the labeling of $S$ is uniquely determined by the values on its leafs.
This can be proved by induction on the rank of $S$. Since the label of the  leafs of $S$ are uniquely determined by $(x,T)$, we can conclude that for each $T\in WF\mathcal{L}T_c$ and every $x\in 2^{\mathbb{N}}$ there is exactly one $S$ such that $(x,T,S)\in G$.

\begin{claim}
The set $G$ is Borel.
\end{claim}
\begin{proof}
We verify that each condition is Borel.
The first and the third conditions are independent of the first coordinate and are closed.

For the second condition. Let $P_s:=\{T\in \mathcal{L}T_c:\ s\mbox{ is a leaf of $T$}\}$ and $Q_s:=\{T\in \mathcal{L}T:(s,1)\in T\}$ for each $s\in \mathbb{N}^{<\omega}$. Then $P_s$ and $Q_s$ are easily seen to be closed. Define 
$$
R_s:=(2^\mathbb{N}\times (\mathcal{L}T_c\setminus P_s)\times \mathcal{L}T)\cup (t_{s(|s|-1)}\times P_s\times Q_s)\cup ((2^\mathbb{N}\setminus t_{s(|s|-1)})\times P_s\times (\mathcal{L}T\setminus Q_s)).
$$
Then  $\bigcap_{s\in \mathbb{N}^{<\omega}}R_s$ is the collection of all $(x,T,S)$ satisfying the second condition. 

The fourth condition is also independent of the first coordinate and one can verify that 
$$
Q'_s:=\{S\in \mathcal{L}T:(s,1)\in T\iff \exists n\in \mathbb{N} (s^\frown(n),1)\in S\}
$$
is Borel. Combination of $P_s$, $Q'_s$ and their complements gives us the desired result.  
\end{proof}

For each $(s,i)\in T$, let $T_{(s,i)}:=\{(t,j):(s^\frown t,j)\in T\}$. Consider the following continuous bijection $\Gamma:\mathcal{L}T_c\to\mathcal{L}T_c$ where 
\begin{itemize}
\item if $(\emptyset,0)\in T$ then $\Gamma(T)=R$ where $(\emptyset,1)\in R$ and $T_{(\emptyset,0)}=R_{((0),0)}$,
\item if $(\emptyset,1)\in T$ then $\Gamma(T)=R$ where $(\emptyset,0)\in R$ and $T_{((0),0))}=R_{(\emptyset,0)}$.
\end{itemize}
In other words, $\Gamma\upharpoonright WF\mathcal{L}T_c$ is the bijection switching the codes for a set and its complement.

\begin{claim}
\label{unique}
Let $T\in WF\mathcal{L}T_c$ and $x\in 2^\mathbb{N}$ then $|\{S:(x,T,S)\in G\}|=1$ and for the unique $(x,T,S)\in G$ we have that  $(\emptyset,1)\in S$ if and only if $x$ is in the set coded by $T$. Moreover, let $(x,T,S),(x,\Gamma(T),S')\in G$, then $(\emptyset,1)$ is in $S$ or $S'$ but not in both of them.
\end{claim}
\begin{proof}
This follows from the discussion after the Definition \ref{G} and the definition of $\Gamma$.
\end{proof}

\begin{proof}[Proof of Lemma\ref{CDcoanalytic}]
Let $G_i:=\{(x,T,S)\in G:(\emptyset,i)\in S\}$ for $i\in 2$. One can easily see that $G=G_0\cup G_1$ and both sets are Borel.  Let $proj(G_i):=\{(x,T):\exists S\in \mathcal{L}T \ (x,T,S)\in G_i\}$. Then from Claim \ref{unique} we have
\[
\mathcal{C}_1= (2^\mathbb{N}\times WF\mathcal{L}T_c)\cap proj(G_1)
\]
and 
\[
\mathcal{C}_0= (2^\mathbb{N}\times WF\mathcal{L}T_c)\cap proj(G_0).
\]
Finally,  we show that the set $(2^\mathbb{N}\times WF\mathcal{L}T_c)\cap proj(G_i)$ is $\mathbf\Pi^1_1$ for $i<2$.
This follows from the classical result that  if $A\subseteq X\times Y$ is Borel, then $\{x\in X:\exists!y\in Y (x,y)\in A\}$ is $\mathbf\Pi^1_1$.  But we can also give a direct proof as follows. 

The sets $H_i:=(2^\mathbb{N}\times \mathcal{L}T_c)\setminus proj(G_i)$ are clearly $\mathbf \Pi^1_1$ and so are $M_i:=WF\mathcal{L}T_c\cap H_i$ for $i<2$. But then using the Claim \ref{unique} we see that $(2^\mathbb{N}\times WF\mathcal{L}T_c)\cap proj(G_i)=M_{1-i}$.
\end{proof}

Next we define a coding of Borel functions from $2^\mathbb{N}$ to $2^{\mathbb{N}}$. Let 
$$
C_n:=\{x\in 2^\mathbb{N}:x(n)=1\}.
$$
Let $f:2^\mathbb{N}\to2^\mathbb{N}$ be a Borel function and let $A_n:=f^{-1}(C_n)$. Then $f$ is described by the sequence $\{A_n\}_{n\in\omega}$ because $f(x)(n)=1$ if and only if $x\in A_n$. Thus the following is the natural definition of codes for Borel functions.

\begin{definition}
Let $\mathcal{F}T=(\mathcal{L}T_c)^{\omega}$ and $WF\mathcal{F}T=(WF\mathcal{L}T_c)^{\omega}$.
\end{definition}

The product topology on  $\mathcal{F}T$ is Polish and  $WF\mathcal{F}T\subseteq \mathcal{F}T$ is $\mathbf \Pi^1_1$.
We denote the elements of $\mathcal{F}T$ also by $T$ and the $n$-th element of $T$ as $T(n)$.

\begin{lemma}
The set $WF\mathcal{F}T$ codes Borel functions from $2^\mathbb{N}$ to $2^\mathbb{N}$ i.e. every sequence $T\in WF\mathcal{F}T$ is a code for a function $f_T$ and for every Borel function $f$ there is a sequence $T\in WF\mathcal{F}T$ such that $f_T=f$.
\end{lemma}
\begin{proof}
As it was mentioned above,  every Borel function $f$ is coded by a sequence of Borel sets $\{A_n\}_{n\in \omega}$. Let $T=(T(n))_n$ be such that  $T(n)\in WF\mathcal{L}T_c$ codes $A_n$ for each $n\in \mathbb{N}$.
\end{proof}

\subsubsection{Coding of selectors and $F_\sigma$ ideals}

Now we  will show that the codes for $F_\sigma$ ideals with Borel selector is $\mathbf \Sigma^1_2$ and then conclude with the main results of this section. 

Consider the following map $\Omega:2^\mathbb{N}\times WF\mathcal{F}T\to 2^\mathbb{N}$ by $\Omega(x,T)(n)=1$ if and only if $x$ is in the set coded by $T(n)$.  From the definitions of $\mathcal{C}_i$,  $\Omega$ and  Lemma \ref{CDcoanalytic} the following  is straightforward. 

\begin{lemma}\label{OmegaBorel}
Let $\mathcal{R}\subseteq \cantor\times \mathcal{F}T\times \cantor$ be given by $(x,T,y)\in \mathcal{R}$ if and only if 
\[
\forall n\in \mathbb{N}\,[ \,(\,(x,T(n))\in \mathcal{C}_1\rightarrow y(n)=1)\,\wedge\, (\,(x,T(n))\in \mathcal{C}_0\rightarrow y(n)=0)\,].
\]
Then $\mathcal{R}$ is $\mathbf\Sigma_1^1$ and for all $(x,T,y)\in \cantor\times WF\mathcal{F}T\times \cantor$ we have
\[
\Omega(x,T)=y \Longleftrightarrow \, (x,T,y)\in \mathcal{R}.
\]
\qed
\end{lemma}

Consider the following set $\mathcal{M}\subseteq 2^\mathbb{N}\times \mathcal{F}T\times K(2^\mathbb{N})$ defined by $(x,T,K)\in \mathcal{M}$ if and only if 
\begin{itemize}
\item $T\in WF\mathcal{F}T$,
\item $\Omega(x,T)\in \downarrow K$,
\item $\Omega(x,T)\subseteq x$,
\item if $|x|=\omega$, then $|\Omega(x,T)|=|x|$.
\end{itemize}

\begin{lemma}
$\mathcal{M}$ is a $\mathbf \Pi^1_1$  subset of $2^\mathbb{N}\times \mathcal{F}T\times K(2^\mathbb{N})$.
\end{lemma}
\begin{proof} 
It follows from Lemma \ref{OmegaBorel}. For instance, the second condition can be expressed as follows: 
\[
T\in WF\mathcal{F}T\,\wedge \,\Omega(x,T)\in \downarrow K \Longleftrightarrow T\in WF\mathcal{F}T\wedge \forall y\in \cantor ((x,T,y)\in \mathcal{R}\rightarrow y\in \downarrow K). 
\]
\end{proof}

\begin{theorem}
The set of all $K\in K(2^\mathbb{N})$ that have a Borel pseudo-selector is $\mathbf\Sigma^1_2$.
\end{theorem}
\begin{proof}
This set may be described as 
$$\{K\in K(2^\omega):\exists T\in \mathcal{F}T\,\forall x\in 2^\omega (x,T,K)\in \mathcal{M}\}$$
which is $\mathbf\Sigma^1_2$.
\end{proof}

\begin{theorem}
\label{fsig-no-selector}
There is a $F_{\sigma}$ tall ideal without a Borel selector.
\end{theorem}
\begin{proof}
The codes of $F_{\sigma}$ ideals with a Borel selector are clearly a subset of all tall $F_\sigma$ ideals and the former set is $\mathbf\Sigma^1_2$ but the later is $\mathbf\Pi^1_2-$complete.
\end{proof}

\begin{corollary}\cite{Kechris94}
There is a closed subset of $A\subseteq \baire\times \baire$ such that $\baire = proj(A)=\{x\in \baire:\exists y\in \baire \ \operatorname{s.t.} \ (x,y)\in A\}$ and it does not have a Borel uniformization.
\end{corollary}
\begin{proof}
The space $X:=2^\mathbb{N}\setminus \{x:\exists n \ \operatorname{s.t.} \ \forall m>n \ x(m)=0\}$ is homeomorphic to $\baire$. The restriction of the relation $S=\{(x,y)\in 2^\mathbb{N}\times 2^\mathbb{N}:x\supseteq y\}$ to $X$ is closed in $X$. By our theorem there is a tall $K\in K(2^\mathbb{N})$ without Borel selector. Then $K\cap X$ is closed in $X$ and the closed set $A:=S\upharpoonright (X\times X)\cap (X\times (K\cap X))$ has no Borel uniformization.
\end{proof}

\subsubsection{Galvin's theorem}
\label{uniformRamsey}

Now we use the  previous result to simply observe that there is no uniform version of Galvin's theorem.

\begin{theorem}
There is $\mathcal{F}\subseteq [\mathbb{N}]^{<\omega}$ such that there is no Borel function $S:2^{\mathbb{N}}\to 2^{\mathbb{N}}$ satisfies $S(x)\in hom(\mathcal{F})$, $S(x)\subseteq x$ and $|S(x)|=\omega$ for every infinite $x\in 2^{\mathbb{N}}$. 
\end{theorem}
\begin{proof}
Combine Theorem\ref{fsig-no-selector} and Proposition\ref{rep-of-closed}.
\end{proof}

\subsection{A $\mathbf \Pi^{1}_2$ tall ideal without a closed tall subset}

We construct a $\mathbf \Pi_2^1$ tall ideal which does not contain $hom(\mathcal{F})$ for every $\mathcal{F}\subseteq [\mathbb{N}]^{<\omega}$.
Recall that $hom(\mathcal{F})$ is $\mathbf\Pi^1_1$ for every $\mathcal{F}\subseteq [\mathbb{N}]^{<\omega}$ and therefore we have the following.

\begin{observation}
\label{uniformR} Let  $R\subseteq 2^{[\mathbb{N}]^{<\omega}}\times [\mathbb{N}]^{\omega}\times \infiN$ be defined by 
\[	
R(\F, x, y)\Leftrightarrow y\subseteq x \; \& \; y\in hom(\F).
\]  
Then $R$ is  $\mathbf\Pi^1_1$.
\end{observation}

\begin{lemma} \cite[Lemma 4.6]{HMTU2017}
\label{psi} There is a continuous function $\psi:\infiN\times\cantor\rightarrow \cantor$ such that for every infinite $x\in \infiN$, the collection $\{\psi(x,y):\; y\in\cantor\}$ is an almost disjoint family of infinite subsets of $x$. Moreover, for all infinite $x$ there is an infinite $z\subseteq x$ such that $z\cap \psi(x,y)=\emptyset$ for all $y\in \cantor$.
\end{lemma}

\begin{theorem}
\label{noclosedtall}
There is a $\mathbf \Pi^{1}_2$ tall ideal $\I$ such that for all $x\in\I^+$ and all $\F\subseteq\finN$ there is $y\subseteq x$ with $y\in hom(\F)\cap \I^+$. In particular, $\I$ does not contain any  closed hereditary tall  set. 
\end{theorem}

\proof The construction is similar to  that  presented in \cite[Theorem 4.7]{HMTU2017}. We will  sketch  the argument below. Let  $\varphi:\cantor\rightarrow 2^{\finN}$ be a continuous
surjection.  By the classical uniformization theorem \cite{Kechris94}, let $R^*\subseteq R$ be a $\mathbf \Pi^1_1$  uniformization for the relation $R$ given by \ref{uniformR}.  Let $\psi$ be given by Lemma \ref{psi}. 
Let
\[
\mathcal{C}_1 = \{y\in \infiN: \; \exists x\in \cantor, R^*(\varphi(x), \psi(\mathbb{N}, x), y)\},
\]
\[
\mathcal{C}_{n+1} = \{y\in \infiN: \exists x\in \cantor, \exists z\in \mathcal{C}_n, R^*(\varphi(x), \psi(z, x), y)\}.
\]
Then  each $\mathcal{C}_n$ is $\bf\Sigma^1_2$. Finally,
let
\[
x\in \mathcal{H}\Leftrightarrow \;(\exists n\in \mathbb{N})\; (\exists y\in
\mathcal{C}_n) \; y\subseteq^* x.
\]
The proof of Theorem 4.7 in \cite{HMTU2017} shows that  $\I=\mathcal{P}(\mathbb{N})\setminus\mathcal H$ is
a tall ideal. We will show that it satisfies the other requirements.   It is clearly $\mathbf\Pi_2^1$.  Let $\F\subseteq\finN$ and $y\not\in \I$. Then there is $x\in \cantor$ such that $\F=\varphi(x)$. There is also $n\in \mathbb{N}$ and  $z\in \mathcal{C}_n$ so that  $z\subseteq^* y$.  Let $w$ be such that 
$R^*(\varphi(x), \psi(z, x), w)$. Then $w\subseteq z$ and is $\F$-homogeneous.  By definition, $w\in \mathcal{H}$. Then $w\cap y$ is infinite and $\F$-homogeneous.

The last claim follows from Lemma \ref{rep-of-closed}.
\endproof

A corollary of the proof of the previous theorem  provides a more general construction of co-analytic tall ideals as in \cite{HMTU2017}. 
 
\begin{theorem}Let $\B$ be a front over $\mathbb{N}$. There is a co-analytic tall ideal $\I$ such that $hom(\F)\not\subseteq \I$ for all $\F\subseteq \B$. 
\end{theorem}

\proof  From the proof of Theorem \ref{noclosedtall} and using  Corollary \ref{uniform-galvin} instead of the co-analytic uniformizing set $R^*$, we define the sets  $\,\mathcal{C}_n$,  which now are analytic. Thus the ideal constructed is co-analytic. 
\endproof 

\begin{question} Does every analytic tall ideal $\I$ contain a $\fsig$
tall ideal? (M. Hrusak).
	
A weaker version of the previous question is for which tall families $\C$  there is $\mathcal{F}\subseteq [\mathbb{N}]^{<\omega}$ such that $hom(\mathcal{F})\subseteq \C$ (here $hom(\mathcal{F})$ is not necessarily  closed).
\end{question}

\section{Examples of  tall families with a Borel selector}

We present some examples showing that the search for a Borel selector can be reduced, in some instances, to find an appropriated coloring. 

%
%
%
%

\begin{ejem}
\label{enumerating-function-2}
	
Let $e:\infiN\rightarrow \baire$ be the increasing  enumeration function, i.e.  $e(x)(n)$ is the nth element of $x$ in its natural order. Notice that $e$ is continuous.  Let
$\gamma:\infiN\times\infiN\rightarrow \infiN$ be given by
\[
\gamma(x,y)=\{e(x)(n):\; n\in y\}.
\]
Notice that $\gamma(x,y)\subseteq x$ and  $\gamma$ is continuous. For each $y\in \infiN$, let
\[
\C_y=\{\gamma(x,y):\; x\in \infiN\}.
\]
Then $\C_y$ is a tall family and obviously $S(x)=\gamma(x,y)$ is a  Borel selector for $\C_y$.
	
We will show that $\C_y$ contains $hom(c)$ for some coloring of pairs $c$. Let $(y_n)_n$ be the increasing enumeration of $y$. We assume that $y_0\geq 1$. If $(z_n)_n$ is the increasing enumeration of an infinite  set $z$, then
\[
z\in \C_y \;\;\Leftrightarrow\;\; (\forall n) (y_{n+1}-y_{n}\leq
z_{n+1}-z_{n}) \;\&\; y_0\leq z_0
\]
Consider the following coloring:
\[
c\{k,l\}=0\;\;\;\mbox{iff}\;\;\; l-k\geq y_k \;\&\;\; k\geq y_0.
\]
Then any $c$-homogeneous infinite set is necessarily $0$-homogeneous. Let $h=\{h_n:\;n\in\mathbb{N}\}\in hom(c)$. Then $h_0\geq b_0\geq 1$ and $h_{k}\geq k+1$. Hence $h_{k+1}-h_k\geq b_{h_k}\geq b_{k+1}\geq b_{k+1}-b_k$. Thus $h\in \C_y$.
\end{ejem}

\begin{question}
{\em Let $K\subseteq\infiN\times\infiN$ be a closed set without a Borel
uniformization (see \cite{Kechris94}). Consider the
following family:
\[
\C=\{\gamma(x, y):\; (x,y)\in K\}.
\]
Since the projection of $K$ is $\infiN$, then  $\C$ is tall
and, by definition, is analytic. We do not know if $\C$ has a
Borel selector.}
\end{question}

\medskip

\begin{ejem}
\label{tall given by ideals}	Let $WO(\mathbb{Q})$ be the collection of all well-ordered subsets of $\mathbb{Q}$ respect the usual order. Let  $WO(\mathbb{Q})^*$  the collection of well ordered subsets of $(\mathbb{Q},<^*)$ where $<^*$ is the reversed order of the usual order of $\mathbb{Q}$. Let ${\C}= WO(\mathbb{Q})\cup WO(\mathbb{Q})^*$. Notice that $\C$ is a complete co-analytic set. To see that $\C$ has a Borel selector, let $c:[\omega]^{2}\rightarrow 2$ given by $c\{r_n,r_m\}= 0$ iff $n<m$, where $(r_n)_n$ is any fixed enumeration of $\mathbb{Q}$. Then $hom(c)\subseteq\C$ and the result follows from corollary \ref{efective-Ramsey}.  
\end{ejem}
\medskip

Let $K$ be a sequentially compact space and $(x_n)_n$ be a
sequence on $K$. Let
\[
\C(x_n)_n= \{y\subseteq \mathbb{N}:\; (x_n)_{n\in y}\;\mbox{is convergent}\}.
\]
Then $\C(x_n)_n$ is tall.

\begin{theorem}
\label{hom-cauchy-sequences}
Let $(x_n)_n$ be dense in a compact metric space $X$. There is a coloring $c$ such that $hom(c)\subseteq \C(x_n)_n$, thus $\C(x_n)_n$  has a Borel selector.
\end{theorem}

\proof Let $f:\cantor \rightarrow X$ be a continuous surjection.
Let $y_n\in\cantor$ such that $f(y_n)=x_n$ for each $n\in \omega$.
Let $\preceq$ be the usual lexicografic order on $\cantor$. Consider the
Sierpinsky coloring $c\{n,m\}_< = 1$ iff $y_n\preceq y_m$. Then
for any $h\in hom(c)$, $(y_n)_{n\in h}$ is a monotone
sequence in $\cantor$ and therefore it is convergent and so is
$(x_n)_{n\in h}$. Hence $hom(c)\subseteq \C(x_n)_n$.

\qed

Now we will look at the ideal of nowhere sets in countable spaces. 

\begin{theorem} Let $(X,\tau)$ be a regular space without isolated points over a countable set $X$. There is a coloring $c: [X]^{2}\rightarrow 2$ such that $hom(c)\subseteq nwd(X,\tau)$  and thus $nwd(X,\tau)$ has a Borel selector. 
\end{theorem}

\proof The Sierpinski coloring $c$ on $[\mathbb{Q}]^{2}$ satisfies that $hom(c)\subseteq nwd(\mathbb{Q})$. Notice that every $c$-homogeneous set is a discrete subset of $\mathbb{Q}$.  Let $(V_n)_n$ be a countable collection of $\tau$-open sets that separates points. Let $\rho$ be the topology generated by the $V_n$'s.  Then $(X,\rho)$ is homeomorphic to $\mathbb{Q}$. Therefore the Sierpinski coloring  on $\mathbb{Q}$ can be defined on $[X]^{2}$ such that every
$c$-homogeneous set is a $\rho$-discrete subset of $X$. Since $\rho\subseteq \tau$, then $hom(c)\subseteq nwd(X,\tau)$.
\endproof

{\bf Acknowledgment:} The second author would like to thank Stevo Todor\v{c}evi\'c for the discussions we had years ago about some of the results presented in this article and for point out to the paper \cite{Solovay78}.

\end{document}